\documentclass[12pt]{amsart}

\usepackage[pdftex]{graphicx}

\usepackage[T2A]{fontenc}
\usepackage[utf8]{inputenc}
\usepackage[english]{babel}

\usepackage{amsmath}
\usepackage{amssymb}
\usepackage{amscd}
\usepackage{amsfonts,amsmath,amsthm}

\usepackage[matrix,arrow,curve]{xy}
\usepackage{xfrac}
\usepackage{color}

\newtheorem{theorem}{Theorem}
\newtheorem{prop}{Proposition}
\newtheorem{cor}{Corollary}
\newtheorem{lemma}{Lemma}
\newtheorem*{main lemma}{The main lemma}

\newtheorem{definition}{Definition}

\title[Each polytope in $\mathbb{R}^3$ has a point with ten normals]{Each  generic polytope in $\mathbb{R}^3$ has a point with ten normals to the boundary }

\author{Ivan Nasonov, Gaiane Panina}

\address{I. Nasonov: St. Petersburg State University, St. Petersburg department of Steklov institute of mathematics  wanua-nasonov-i04@yandex.ru; G. Panina: St. Petersburg department of Steklov institute of mathematics
 gaiane-panina@rambler.ru  }
\date{\today}

\begin{document}

\begin{abstract}It is conjectured since long that each smooth convex body $\mathbf{P}\subset \mathbb{R}^n$ has a point in its interior   which belongs to at least $2n$ normals from different points on the boundary of $\mathbf{P}$. The conjecture is proven for $n=2,3,4$.

We treat an analogous problem for convex polytopes in $\mathbb{R}^3$  and prove that each generic  polytope  has a point in its interior with at least $10$ normals to the boundary.  

This  bound is exact: there exists a tetrahedron with no more than $10$ normals emanating from a point in its interior.

The proof is based on piecewise linear analog of Morse theory, analysis of bifurcations, and some combinatorial tricks.
\end{abstract}

\maketitle

\section{Introduction}

\medskip

It is conjectured since long that for any smooth convex body $\mathbf{P}\subset \mathbb{R}^n$ there exists a point in its interior   which belongs to at least $2n$ normals from different points on the boundary of $\mathbf{P}$. The conjecture is a simple exercise for $n=2$. There are two proofs for  $n=3$:  a geometrical proof by E. Heil \cite{Heil},  and a topological proof by J. Pardon \cite{Pardon}. 

So, each $3$-dimensional smooth body has a point with at least $6$ emanating normals. This is an exact bound: an ellipsoid (with all different lengths of axes) has at most $6$ normals from a point in its interior.

 For   $n=4$, the conjecture was proven by J. Pardon \cite{Pardon}, and (to the best of our knowledge) nothing is known for higher dimensions.

In this paper we solve a similar problem for  compact convex polytopes with non-empty interior. 

\begin{definition}  Let $\mathbf{P}\subset \mathbb{R}^3$  be a compact convex polytope with non-empty interior, let $y\in Int~\mathbf{P}$. 
 Let h be a support plane to $\mathbf{P}$ and a point $z\in h\cap \mathbf{P}$  be such that the line $yz$ is orthogonal to $h$.
Then  $yz$ is called a \textit{normal to the boundary of $\mathbf{P}$ emanating from }$y$. The point $z$ is called the\textit{ base }of the normal.
\end{definition}

Our main result is:

\begin{theorem}\label{Thmmain}
  Each generic\footnote{Genericity is defined in Definition \ref{DefGen}.}  polytope in $\mathbb{R}^3$ has a point with (at least) $10$ emanating normals to the boundary.
\end{theorem}

 Three necessary remarks are:
 \begin{itemize}
   \item In the smooth case the bases of the normals to $\partial \mathbf{P}$ emanating from $y$ are  the critical points of the \textit{squared distance function } $$SQD_y(x)=|x-y|^2:\partial \mathbf{P}\rightarrow \mathbb{R}.$$ Generically,  $SQD_y$ is a Morse function, thus Morse theory gives a tool.
With some efforts, Morse-theoretic machinery  extends also to polytopes \cite{NasPanSiersma}.  So counting normals to the boundary is equivalent to counting critical points of $SQD_y$.
   \item Each polytope in $\mathbb{R}^3$ has a point with at least $8$ normals to the boundary. This follows  straightforwardly from Morse counts, see \cite{NasPanSiersma}: let $y$ be the center of the sphere of maximal radius contained in  $\mathbf{P}$. Then the function $SQD_y$ has at least four local minima (the tangent points) and at least one maximum, therefore at least three saddles.
   \item The bound in Theorem \ref{Thmmain} is exact: there exists a tetrahedron with no more than $10$ normals from a point in its interior \cite{NasPanSiersma}.
 \end{itemize}
 
 \medskip
 
 The \textbf{main idea in a nutshell} looks like this:  (1)  take a very special point $y$ in the polytope $\mathbf{P}$ with some guaranteed number of emanating normals. (2) Force $y$ to move along a carefully chosen straightline  trajectory. At bifurcation moments od the movement, the emanating normals  born or die. Analyze this birth-death process.  
 
 We use this idea twice: in Proposition \ref{propNice}, where $y$ travels along a special ray emanating from a vertex of $\mathbf{P}$, and  in the proof of the main theorem, where $y$ travels along a special (saddle-type) binormal of $\mathbf{P}$.  
 
Finally, a combinatorial trick puts the pieces together.

\medskip

\newpage

\section{Definitions and preliminaries}\label{SecDefPrelim}

Most of the definitions and lemmata of the section are borrowed from \cite{NasPanSiersma}; we repeat them here for the sake of completeness.
\subsection{Active regions, bifurcation set and its sheets}
Let $\mathbf{P}\subset \mathbb{R}^3$ be a compact convex polytope with non-empty interior.

Let $F$ be a face, an edge, or a vertex of a polytope $\mathbf{P}$. The \textit{active region} $\mathcal{AR}(F)$ is the set of all points $y\in Int\, \mathbf{P}$ such that $F$ contains the base of some normal emanating from $y$. 

The\textit{ bifurcation set}  $\mathcal{B}(\mathbf{P})$ is the union of the boundaries of active regions. It is the piecewise linear counterpart of the focal set that appears in the smooth case.

The bifurcation set is piecewise linear. Its  parts not lying on $\partial \mathbf{P}$ (\textit{sheets}) fall into two types, red and blue, see  Fig. \ref{fig:ActiveRegion}.

\begin{figure}[h]
 \includegraphics[width=0.6\linewidth]{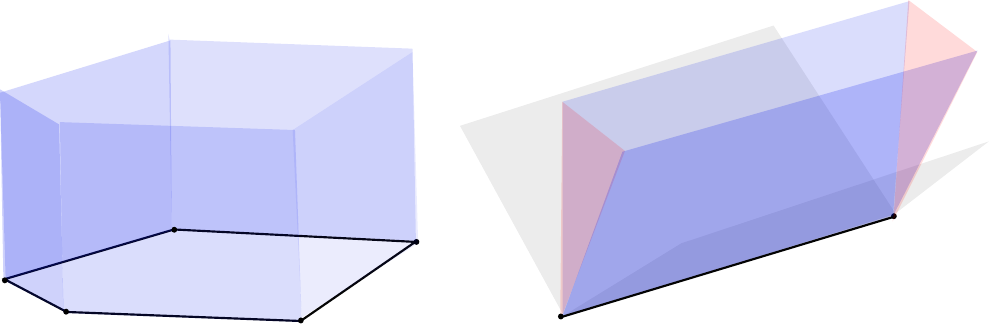}
  \caption { Active regions of a {face and an edge}}
  \label{fig:ActiveRegion}
\end{figure}

\medskip

That is, 
\begin{enumerate}
  \item The active region of a face $F$ is the intersection of the infinite  right angled prism  based on $F$ and  the polytope $\mathbf{P}$.
  
  The side faces  of the prism are called \textit{ blue  sheets} of  $\mathcal{B}(\mathbf{P})$.
  \item The active region of an edge $E$  sharing two faces $F_1$ and $F_2$ is the intersection of  the polytope $\mathbf{P}$ with the wedge-like polyhedron bounded by 
  two planes  containing $E$ and orthogonal to $F_1$ and $F_2$ (these are blue sheets), and another two planes  containing the endpoints of $E$ and orthogonal to $E$ (\textit{red sheets}). 
  \item The active region of a vertex $V$ is the inner normal cone of $\mathbf{P}$ at the vertex $V$ (bounded by red sheets)  intersected with the polytope $\mathbf{P}$. 
\end{enumerate}

\medskip

Although the function $SQD_y$ is  non-smooth,
it can be treated as a Morse function.  So counting critical points  of $SQD_{y}$ amounts to counting normals, exactly as it is in the smooth  case.

\begin{lemma}

For $y \notin \mathcal{B}(\mathbf{P})$, we have:
\begin{enumerate}
  \item Local maxima of $SQD_y$  are attained at some  vertices of $\mathbf{P}$.
  \item Local minima of $SQD_y$ are attained at some faces of $\mathbf{P}$.
  \item Saddles of $SQD_y$ are attained at some  edges of $\mathbf{P}$.
\end{enumerate}

\end{lemma}

\textbf{Convention:}  from now on, we call local maxima (local minima) just maxima and minima for short.

\begin{lemma}\label{PropBifurc}

If $y$ crosses transversally one sheet of $\mathcal{B}(\mathbf{P})$, a pair of critical points of  $SQD_{y}$ either dies or is born.
\begin{enumerate}
  \item Crossing a blue sheet amounts to birth-death of a minimum and a saddle.
  \item Crossing a  red sheet amounts to birth-death of a maximum and a saddle.
\end{enumerate}
\end{lemma}

We say that a point $y\in Int(\mathbf{P})$ \textit{projects} to a face $F$ (to an edge $e$, respectively) if the orthogonal projection of $y$ to the affine hull of $F$ (of $e$, respectively) lies in the interior of the face $F$ (or of the edge $e$).

An edge of a polytope with an acute dihedral angle is called an \textit{acute edge}.

\begin{lemma}\label{LemmaSad}\begin{enumerate}  
                               \item A point $y\in \mathbf{P}$ projects to a face iff the face contributes a minimum of $SQD_y$.
                               \item  Let  $e$ be an acute edge of  $\mathbf{P}$. Then  a  point $y\in \mathbf{P}$ projects to $e$ iff the function $SQD_y$ attains a saddle  at the edge $e$.
                             \end{enumerate}
\end{lemma}

\begin{lemma}\label{LemmaSaddles} For a convex polytope $\mathbf{P}\subset \mathbb{R}^3$ and a  point $y\notin \mathcal{B}(\mathbf{P})$, the number of critical points of $SQD_y$ (equivalently, the number of normals emanating from $y$) equals $2+2s$,
   where $s$ is the number of saddles.
   
   In particular, the number of normals emanating from $y$ is even.

\end{lemma}

In the paper we restrict ourselves to \textit{generic  polytopes}:

\begin{definition}\label{DefGen} A polytope $\mathbf{P}$ with the vertex set $\mathbf{V}$ is generic if  affine hulls of any two (two-element or three-element) subsets of
$\mathbf{V}$ are neither parallel nor orthogonal, unless this is dictated by unavoidable reasons. For example,
  \begin{description}
  \item[i] No two edges  are orthogonal.
  \item[ii]  No edge is parallel to a face, unless they are incident. 
\end{description}
\end{definition}

Important observations are: (1) {For a generic polytope, the sheets of $\mathcal{B}(\mathbf{P})$ intersect transversally.}
(2) {Any polytope can be turned to a generic one by a close to $id$ projective transform.} (3) If $\mathbf{P}$ is generic, then any sheet of $\mathcal{B}(\mathbf{P})$ 
related to a face contains no vertices of $\mathbf{P}$ except for vertices of the face.

\medskip

\subsection{Associated spherical geometry}

A \textit{spherical polygon}  is a polygon with geodesic edges lying in the standard sphere $S^2$. It is always supposed to fit in an open hemisphere.

A vertex $V$ of a polytope $\mathbf{P}$ yields a convex spherical polygon $Spher(V)$  which is the intersection of a small sphere centered at $V$  with  the polytope $\mathbf{P}$.  Although the sphere is small, let us assume that it is equipped with the metric of the unit sphere  $S^2$.

For points $X,Y \in S^2$ denote by  $|XY|$ the spherical {distance} between them.

 Say that  a point  $Y$ \textit{projects to }a geodesic segment $s$  if (i) the segment and the point $Y$ fit in an open hemisphere, (ii) $s$ contains a point  $W$ such that the geodesic segment $YW$ is orthogonal to $s$, and (iii) $|YW|<\pi/2$.  In this case  $|YW|$ realizes the Hausdorff distance between the point  $Y$ and the segment $s$.

Let $P\subset S^2$ be a convex spherical polygon, let $Y\in Int~P$. Define the spherical squared distance function
$$SQD_Y^{sph}:  \partial P \rightarrow \mathbb{R}, \ \ SQD_Y^{sph}(Y)=|XY|^2.$$

Let $V$ be a vertex of a polytope $\mathbf{P}$. Let a point $y$ lie in the interior of $\mathbf{P}$, close to the vertex $V$. Here and in the sequel we denote by $Y$ the intersection of the small sphere centered at $V$ and the ray $Vy$.

\begin{definition}
 
 The core   of a convex spherical polygon $P$ is defined as follows:
 $$Core(P)=\{Y|  \ \ Y \hbox{ lies  in the interior of $P$, and } |YX|\leq \pi/2 \ \ \forall  X \in \partial P\}.$$
\end{definition}

\begin{lemma}\label{LemmaCriterionMaxMin Sadd}
  In the above notation,  \begin{enumerate}
                          \item $V$ is a  maximum of   $SQD_y$  iff
                         
                           $Y \in Core(Spher(V))$.
                           
                          \item Let $e$ be an edge of $\mathbf{P}$ incident to $V$. Assume  $A= Spher(V)\cap e$. The edge $e$ contributes  a saddle point of   $SQD_y$  iff \begin{enumerate}
                                                                                           \item $A$ is a  maximum of $SQD_Y^{sph}$, and
                                                                                           \item $|AY|<\pi/2$.
                                                                                         \end{enumerate}
                          \item A  face $F$  incident to $V$ contributes  a minimum of   $SQD_y$  iff  $Y$ projects to the edge  $F\cap S^2$ of $Spher(V)$.
                                                                                           
                        \end{enumerate}
\end{lemma}

\begin{definition} \label{dfnNicePoly} 

\begin{enumerate}
  \item A point $Y\in Int(P)$ has a short maximum (respectively, short minimum) at $X\in \partial P$ if $X$ is a  maximum (respectively, a  minimum) of $SQD^{sph}_Y$, and $|YX|<\pi/2$.

  \item  A convex spherical polygon $P \subset S^2$ is nice if there is a point $Y$ in its interior such that $SQD^{sph}_Y$ has at least three  short maxima. Otherwise  $P$ is called skew.

  \item  A vertex $V$ of a polytope $\mathbf{P}\subset \mathbb{R}^3$ is skew if the associated  spherical polygon $Spher(V)$ is skew. Otherwise $V$ is called nice.
\end{enumerate}

\end{definition}

Short maxima are attained at (some) vertices, whereas short minima are attained at edges.  (This is no longer true for non-short maxima and minima.)

For a fixed $Y$, maxima and minima of $SQD_Y^{sph}$  alternate on $\partial P$. We say that a minimum and a maximum are \textit{neighbours}  if there is no other maxima and minima on a connecting arc of $\partial P$.

\medskip
\textbf{Example.}  Skew triangles were studied in \cite{NasPanSiersma} in detail. 
 A skew triangle is depicted in Fig \ref{fig:skew}.

\begin{figure}[h]
  \includegraphics[width=0.7\linewidth]{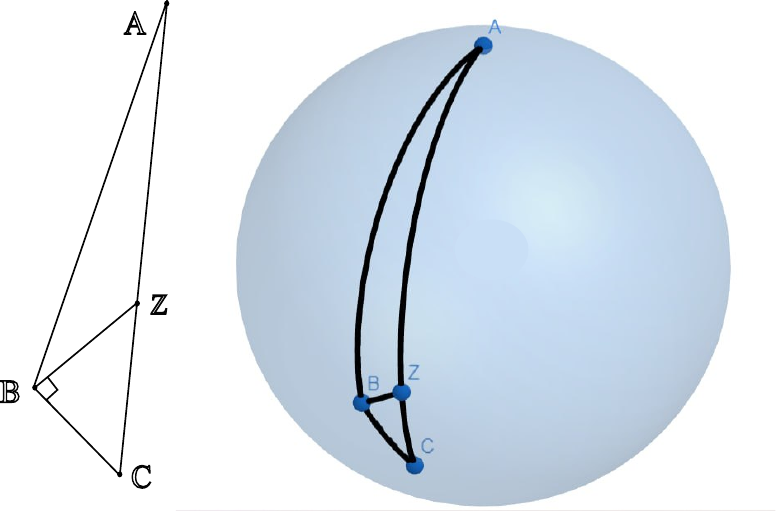}
  \caption { A skew triangle}
  \label{fig:skew}
\end{figure}

\section{Proof of Theorem 1}\label{SecProofThm1}
\begin{prop} \label{propNice}

If a generic polytope $\mathbf{P}\subset \mathbb{R}^3$ has no point in its interior with 10  emanating normals, then all the vertices of  $\mathbf{P}$ are skew.

\end{prop}

\begin{proof}

Assume the contrary: $V$ is a nice vertex of $\mathbf{P}$, and $\mathbf{P}$ has no point in its interior with 10  emanating normals. Let $Y$ be a point inside $P = Spher(V)$ with three short maxima $A_1, A_2,A_3$. Each of the maxima is surrounded on $\partial P$ by two short minima, so there are at least three short minima of $SQD^{sph}_Y$.  Consider two cases. 

\begin{enumerate}
    \item The number of the short minima (that are neighbors of $A_i$) is exactly $3$. Then all the minima are short, 
 $Y\in Core(Spher(V))$, so $Y\in \mathcal{AR}(V)$.

     Therefore,
for every point on the ray $VY$, the vertex $V$ is a  maximum. Take a point $y$ lying
on the ray very close to $V$ and let it travel along the ray. We may assume that $y$ never crosses two sheets of the bifurcation set at a time. If this is not the case, perturb $y$ a little bit. 

 Initially $SQD_y$ 
has  three minima $m_1$, $m_2$, and $m_3$, lying on the faces incident
to $V$ , three saddles $s_1$, $s_2, s_3$, lying on edges that correspond to $A_1, A_2, A_3$, and one maximum $V$.  Since $V$ is not
 the global maximum, there is at least one other maximum. Altogether there are at least 8 critical points of $SQD_y$, or, equivalently,
8 normals emanating from $y$.

The point $y$ keeps moving until the \textit{first event},  which is either a bifurcation of $SQD_y$,  or  the point $y$ leaves the
interior of $\mathbf{P}$. We have the following case analysis:

(a) The first event is a birth of two new critical points. This gives us 10 normals.

(b)
The point $y$ leaves the polytope $\mathbf{P}$. This means that $y$ hits one
of the faces which is not incident to the vertex $V$. Right before
$y$ leaves the polytope $\mathbf{P}$, there is the  fourth minimum on this face, and we have at least 10 critical points.

(c) One of the saddles $s_i$ dies. This happens in two cases:
(i) $s_i$ meets a  maximum of $SQD_y$. The maximum cannot be $V$ since  $s_i$ slides away from $V$.
It is not any other maximum $V'$. Indeed, recall
that $s_i$ is the orthogonal projection of $y$ to the edge. If $s_i$
meets $V'$ before $y$ leaves $\mathbf{P}$, then $|Vy|>|VV'|$. Then $V'$ can't be maximum of $SQD_y$ when $y=V$.
(ii) $s_i$ meets a minimum. It can be none of $m_1$, $m_2$, and $m_3$, so
right before the bifurcation, the function $SQD_y$ has $4$ minima and $2$ maxima, which proves the claim.

\item The number of short minima that are neighbors of $A_i$ is at least $4$. 

 We repeat the above construction. Initially we have at least four minima of $SQD_y$, three saddles, and  one  maximum.  The bifurcations are analyzed in the same way. If there is no bifurcation until $y$ leaves $\mathbf{P}$, we arrive at a fifth minimum just before leaving the polytope, which gives $10$ normals.

\end{enumerate}

\end{proof}

The following properties   of skew polygons are crucial for the proof:

\begin{prop}\label{propSkew}
\begin{enumerate}
    \item  Each skew polygon $P$ has exactly two acute angles, say, at vertices $M_1$ and $M_2$,
    \item $|M_1M_2|>\pi/2$,
    \item For any $p\in P$, the inequalities $|M_1p|<\pi/2$ and $|M_1p|<\pi/2$ imply that $p \in Core(P)$. 
\end{enumerate}
\end{prop}

We postpone the proof of the proposition to the next section.

\subsection*{Proof of Theorem 1 modulo Proposition \ref{propSkew}}

Assume the contrary: a generic polytope $\mathbf{P}\subset \mathbb{R}^3$ has no point in its interior with 10  emanating normals.
By Proposition \ref{propNice}, all the vertices of $\mathbf{P}$ are skew. Then, by Proposition \ref{propSkew} (1), each vertex of $\mathbf{P}$ is incident to exactly two acute edges. Therefore, acute edges form a closed broken line on the boundary of $\mathbf{P}$. Denote the broken line by $\mathcal{A}$.

For an interior point $y \in \mathbf{P}$  set $$SQD_{y,\mathcal{A}}: \mathcal{A}\rightarrow \mathbb{R}, \ \ SQD_{y,\mathcal{A}} (x) = |x-y|^2.$$ 

\begin{lemma}\label{LemmaSadAcute}
\begin{enumerate}
  \item Minima of  $SQD_{y,\mathcal{A}}$ are attained at edges of $\mathcal{A}$. Each such a minimum is a saddle critical point of \newline $SQD_y: \partial \mathbf{P}\rightarrow \mathbb{R}.$
  \item  Maxima of  $SQD_{y,\mathcal{A}}$ are attained at vertices of $\mathcal{A}$. Maxima of  $SQD_{y,\mathcal{A}}$ are exactly maxima of $SQD_y$.
\end{enumerate}
\end{lemma}
\begin{proof}
  (1) An edge contains a minimum of  $SQD_{y,\mathcal{A}}$ iff $y$ projects to it. Since the edge is acute, by Lemma \ref{LemmaSad} we have a saddle of $SQD_y$.
  
 (2) One way is clear:  a maximum of $SQD_y$ is  also a maximum of $SQD_{y,\mathcal{A}}$.
   The converse follows from Proposition \ref{propSkew}, (3) and Lemma \ref{LemmaCriterionMaxMin Sadd}.
\end{proof}

\begin{definition}
  A binormal of $\mathbf{P}$ is a segment $VW$ with endpoints on $\partial \mathbf{P}$ such that there exist two parallel support planes  orthogonal to $VW$, 
  one passing through $V$, and the other one through $W$.
\end{definition}

\begin{lemma}\label{binormal}
There exists a binormal $VW$ of $\mathbf{P}$ such that $V$ is a vertex, and $W$ lies in the interior of some edge.
\end{lemma}
\begin{proof}

    Binormals of $\mathbf{P}$ correspond to critical points of the width function of $\mathbf{P}$. \footnote{By definition, \textit{width function} of $\mathbf{P}$ equals the distance between two parallel support planes to $\mathbf{P}$. The width function is clearly not a smooth function.
    So our aguments rely on the concept of \textit{Clarke subdifferential} borrowed from non-smooth analysis, see e.g.
 \cite{Agr}, which contains also the definition of critical point and the regular
interval theorem. 
 }

    We have the following case analysis:
    
  (1) If the  endpoints of a binormal are two vertices of $\mathbf{P}$, then the binormal realises a  maximum of the width function.

  (2) If the  endpoints of a binormal are a vertex of $\mathbf{P}$ and a point in the interior of a face, then the binormal realises a minimum.

  (3) If the  endpoints of a binormal lie in the interiors of some (non-parallel, because of genericity) edges of $\mathbf{P}$, then the binormal also realises a  minimum.

  Since the width function is an even function defined on $S^2$, or equivalently, a function defined on the projective plane, it cannot have maxima and minima only. Therefore  there should be other types of binormals. The only remaining generic type is the one we need.
\end{proof}

Now we are ready to prove the theorem. 

Consider a binormal from Lemma \ref{binormal}.  Let its endpoints be $V$ (a vertex of $\mathbf{P}$) and $W$  (a point in the interior of an edge).
Remind that by assumption and Proposition \ref{propNice},  $V$ is a skew vertex (as well as all the other vertices of $\mathbf{P}$.

\bigskip

\textbf{Case 1. The  binormal  $VW$ lies in the interior of $\mathbf{P}$, } Fig. \ref{fig:Binormal}.

Observe that $VW$ belongs to the active region $\mathcal{AR}(V)$, that is, for every point on the binormal, $V$ is a  maximum.

We may assume that the binormal never crosses two different sheets of $\mathcal{B}(\mathbf{P})$ at a time.  If this is not the case, replace the binormal by a generic broken line emanating from $V$ and ending at $W$  lying close to the binormal.  

Take a point $y$ lying on the binormal very close to the vertex $V$. By Lemmata \ref{LemmaCriterionMaxMin Sadd} and \ref{LemmaSad}, 
$SQD_y$    has two saddles $s_1, s_2$ on the acute edges incident to $V$, a saddle $s_3=W$,  two minima $m_1, m_2$ on the faces incident to $V$, and a maximum at $V$. Since $V$ is not a global maximum, there is some other
 global maximum; altogether  there are at least $8$  normals emanating from $y$.

\begin{figure}[h]
 \includegraphics[width=0.4\linewidth]{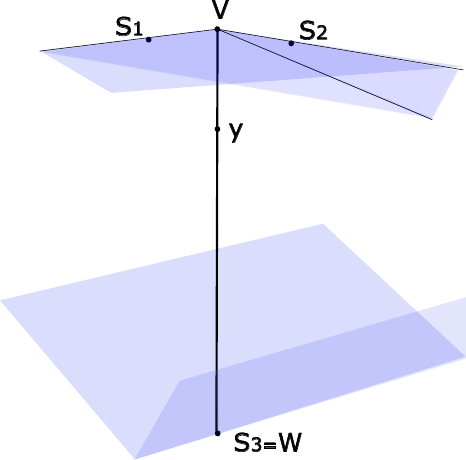}
  \caption { Binormal $VW$}
  \label{fig:Binormal}
\end{figure}

\begin{lemma}\label{LemmaVisNotMax}
 Let $VV'$ and $VV''$  be  two acute edges incident to $V$. For a point $y$ lying close to $V$
 none  of  $V'$ and $V''$ is a  maximum of $SQD_y$.
\end{lemma}\label{LemmaNotMax}
\begin{proof}
  By  Proposition \ref{propSkew}, (2), there is an edge incident to $V'$ which makes an obtuse angle with $VV'$. Therefore $V'$ is not a maximum neither for $V$  
  nor for any point $y$ close to $V$.
\end{proof}

Now let the point $y$ travel along the binormal starting at $V$.
 As we have seen, initially there  are at least emanating $8$ normals.

Keep moving until the  \textit{first  event},  which is either a bifurcation (that is, a crossing of a sheet of $\mathcal{B}(\mathbf{P})$),  or   the point $y$ reaches $W$. We have the following case analysis. \begin{enumerate}

\item The first event is a birth of a new pair of critical points. Then right after bifurcation we have at least 10 normals.

\item The first event is a death of two critical points.
\begin{enumerate}
  \item One of the saddles  $s_i$ meets a maximum, and the two critical points die. This saddle cannot be $s_3$ since $s_3$ does not move. It can be neither $s_1$ nor $s_2$ by Lemma \ref{LemmaVisNotMax}.
\item One of the saddles $s_{1,2}$ meets a minimum. The minimum point is necessarily attained at a face incident to $V$, which is impossible since the distances between the bases of the normals grow as $y$ slides away from $V$.
\item The saddle $s_{3}$ meets a minimum. This happens only when $y$ reaches $W$. This case is analyzed below.
 
\end{enumerate}

\item Point $y$ reaches  $s_3 = W$ without bifurcation. Once $y$ almost reaches $W$, we observe two new minima lying on faces incident to the edge containing  $W$. By condition, no face of $\mathbf{P}$ contains both $V$ and $W$, therefore none of the two  new minima coincides with $m_1$ or $m_2$. Since there were no bifurcations on the way, all the ''old'' critical points persist, and eventually we have $4$ minima, $2$ maxima, and  $3$ saddles. Since the number of critical points is even,    altogether there are  at least 10 normals.
                                           \end{enumerate}

\medskip

\textbf{Case 2.  $VW$ lies in the interior of a face of $\mathbf{P}$.}

In this case the edge containing $W$ is acute.

One repeats the arguments from Case 1 for a point $y$ travelling along the binormal shifted a little bit inwards the polytope.
The same bifurcation arguments survive: if $y$ crosses a sheet of $\mathcal{B}(\mathbf{P})$ on its way, we get $10$ emanating normals.

  So the only remaining case is the one when there are no bifurcations on the (shifted) binormal.
 Here we need additional arguments since one of the two new minima coincides
with either $m_1$ or $m_2$. However we conclude that for a point $y$ on the (shifted) binormal we have (at least) $3$ minima, $2$ maxima, and $3$ saddles.
Moreover, these  $3$ saddles are attained on acute edges of $\mathbf{P}$. That is, we have three minima of $SQD_{y,\mathcal{A}}$, see Lemma \ref{LemmaSadAcute}.  Since $\mathcal{A}$  is a topological circle (or a disjoint union of topological circles),  there are necessarily three maxima of $SQD_{y,\mathcal{A}}$.  By Lemma \ref{LemmaSadAcute}, they are maxima of $SQD_y$ as well.

Now for $y$, we have (at least) $3$ saddles, $3$ maxima  and $3$ minima which gives at least $10$ critical points since the number of critical points is even.

\section{Proof of Proposition \ref{propSkew}}\label{SecProofProp2}

Since the proof is quite technical, we recommend the reader to start with proving the proposition for skew triangles only,   using  the following:

\begin{lemma}\cite{NasPanSiersma}\label{lemSkewTria}

A spherical triangle $ABC$ is skew iff (up to relabeling of the vertices), one has  $|CA|> \pi/2$, 
                                                                                                 $|BC|< \pi/2$,
                                                                                        $\angle ABC> \pi/2$,      
                                                                                                  $\angle BAC <\pi/2$,
                                                                                                 $ \angle BCA < \pi/2$,
                                                                                                  $|BA|> \pi/2$,
$|AZ| > \pi/2$ for the point $Z\in AC$ such that $ZB$ is orthogonal to $BC$.
\end{lemma}

\medskip

Throughout the section $P$ is a generic skew spherical polygon.

\begin{definition}\label{defMinBigon}
     \textit{A bigon} (or a lune) is a spherical polygon with two edges. It is bounded by two halves of big circles. 
     The minimal bigon  $Big(P)$ of a spherical polygon $P$ is the bigon containing $P$ which is minimal by inclusion, or, equivalently, with the minimal angle between its edges.
\end{definition}


\begin{lemma} \label{MinBigonLemma}
    $P$ intersects each edge of its minimal bigon $Big(P)$  by a segment whose interior contains  the midpoint of the edge, see Fig \ref{Fig1}.
\end{lemma}
\begin{proof} Denote the edges of $Big(P)$ by $l_1$ and $l_2$, and their midpoints  by $R$ and $T$.
    Firstly, prove that $R$ and $T$ lie on $\partial P$. Indeed, if $R\notin \partial P$ a rotation of $l_1$ decreases the bigon.

\begin{figure}[h]
  \includegraphics[width=0.4\linewidth]{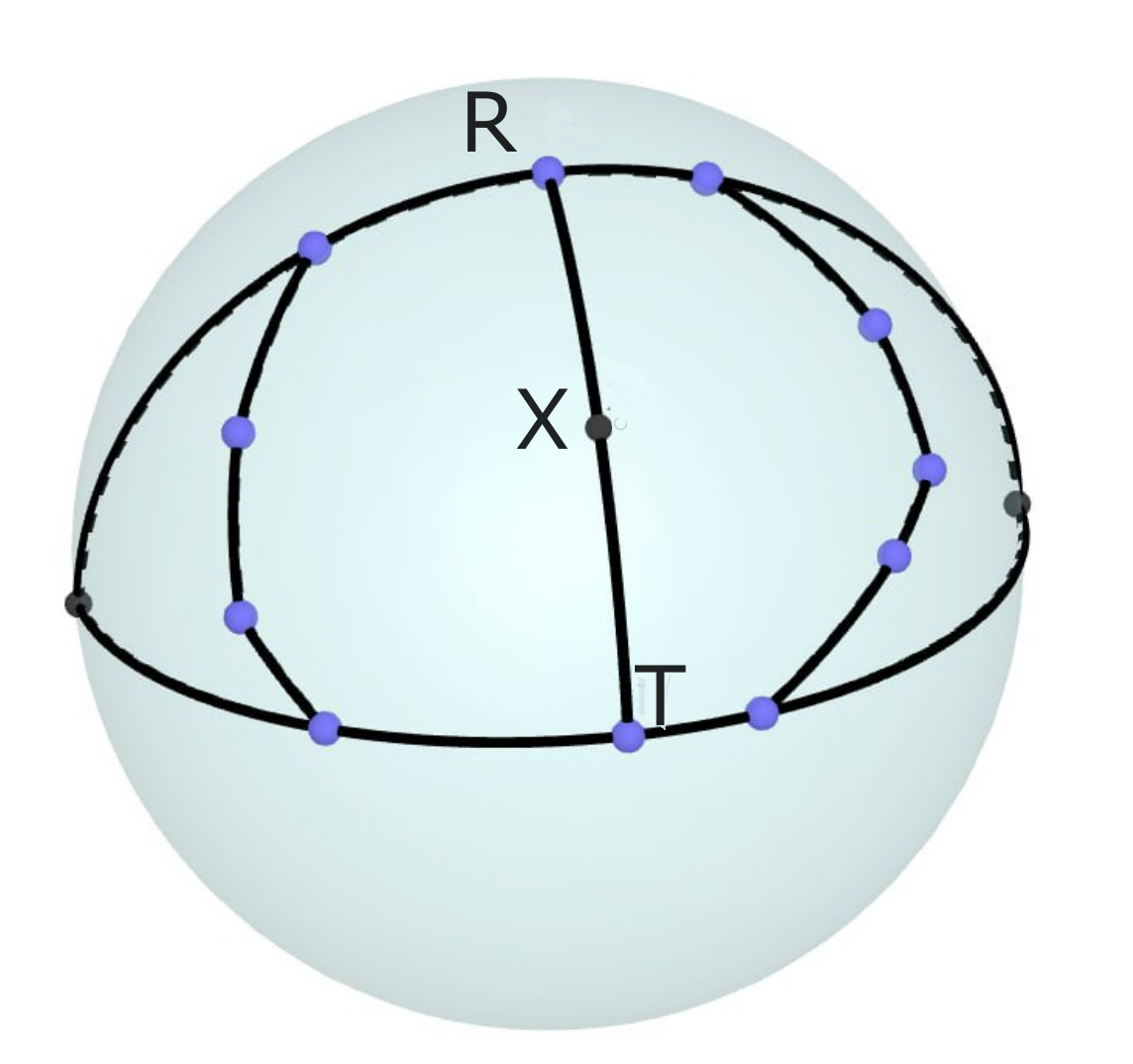}
  \caption {$P$ intersects $Big(P)$  by two segments, containing the midpoints $R$ and $T$. }
  \label{Fig1}
\end{figure}

    Next, prove that $P$ intersects each edge of $Big(P)$ by a segment. If not, two cases are possible:
    
    (1) $P$ intersects each of the edges by a point. Then a rotation  decreases $Big(P)$.

    (2) $P$ intersects one edge by a point, and the other edge by  a segment. Then necessarily the intersection point is the midpoint of the edge, say, $R$, and also the intersection segment contains $T$. Let us show  that  $P$ is not skew. Consider the segment $RT$; it lies in $P$ and cuts the bigon into two equal halves. Let $X$ be the midpoint of $RT$, then $X\in Core(P)$ and $SQD^{sph}_X$ has a short maximum at $R$ and a short minimum at $T$. Since $|XR| = |XT|$,  $SQD^{sph}_X$ has additionally at least two short maxima lying in different halves of bigon. Altogether we have 3 short maxima.  A contradiction. 
    
\end{proof}

\begin{cor} The midpoint $X$  of $Big(P)$ lies in $P$. Besides,  $X\in Core(P)$ and $SQD^{sph}_X$ has two (short) minima at $R$ and $T$. 
\end{cor}

\begin{prop} \label{propAcuteBigon}
   If $P$ is skew, then $Big(P)$ is acute-angled.
\end{prop}

\begin{proof}
    Assume the contrary: $Big(P)$ is obtuse-angled. Then all the angles of $P$ are also obtuse  (this follows from minimality).

    The proof is based on two lemmata.

A spherical segment  $AB$ is \textit{short} if $|AB|<\pi/2$. Otherwise, the segment is called \textit{long}.
\begin{definition}
  Let $A$ be a vertex of $P$. Set $\mathcal{AR}^{spher}(A)$ be the set of points $Y\subset P$ such that $A$ is a short maximum of $SQD_Y^{spher}$.
\end{definition}

    \begin{lemma}\label{lemInciShort}
       If $P$ is skew and $Big(P)$ is obtuse-angled, then no two short edges of $P$ share a vertex.
    \end{lemma}

    \begin{proof}
        Let short edges $AB$ and $AC$ be incident. Let the point $Z$  be as is depicted in Fig. \ref{Fig2.1}.
         That is, $Z$ is the intersection of two perpendiculars to the edges emanating from $B$ and $C$. We have: $|ZB|<\pi/2, |ZA|<\pi/2, |ZC|<\pi/2, Z\in \mathcal{AR}^{spher}(A)$. If  $Z\in P$, then  $P$ is not skew, since for point $Y$ lying close to $Z$, $SQD^{sph}_Y$ has $3$ short maxima. Indeed, shift $Z$ a little bit upwards. 
        
\begin{figure}[h]
  \includegraphics[width=0.4\linewidth]{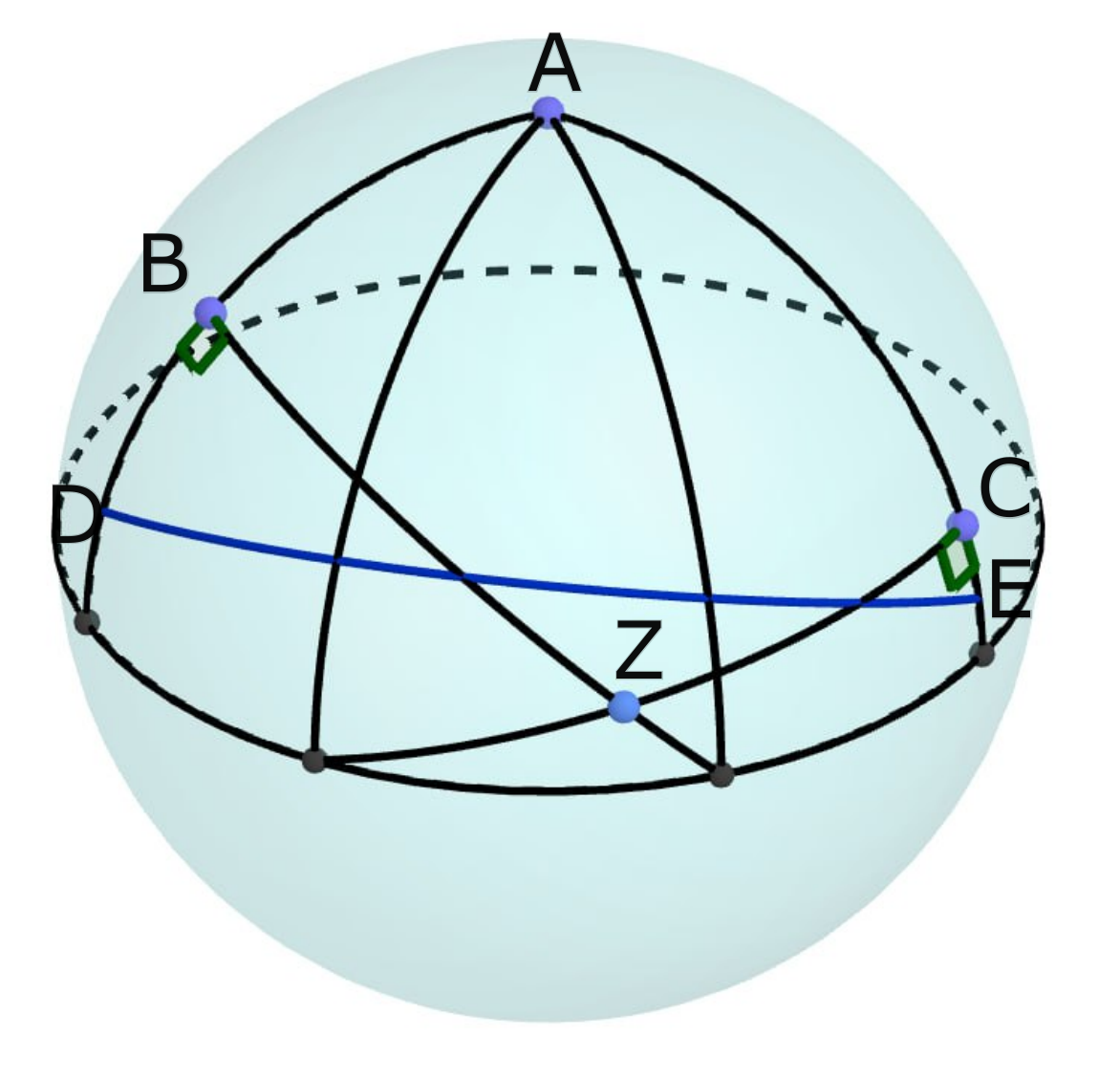}
  \caption {The point $Z$ and the line $l$ for the proof of Lemma \ref{lemInciShort} }
  \label{Fig2.1}
\end{figure}
       
        Assume that $Z\notin P$. Then there is a (geodesic) line $l$ that separates $Z$ and $P$. Consider the intersection points $D$ and $E$ of $l$ with the sides of the angle $A$.

        (1) If $D$ and $E$ lie in the upper hemisphere, then $P$ is contained in a quarter of the sphere, so the minimal bigon cannot be obtuse-angled.

        (2) Let $D$ lie in the bottom hemisphere. All angles in the triangle $ADE$  are obtuse, therefore the intersection point of $l$ with the polar line of $A$ doesn't lie in $\mathcal{AR}^{spher}(A)$. Then $Z\notin \mathcal{AR}^{spher}(A)$. A contradiction.
    \end{proof}

    \begin{lemma} \label{lemAtLeast4}
    If $P$ is skew and $Big(P)$ is obtuse-angled, then $P$ has at most $4$ vertices.
    \end{lemma}

    \begin{proof} The midline $RT$ cuts the bigon into two halves. If $P$ has at least 5 vertices, then by Lemma \ref{lemInciShort} there exists a long edge $AB$, contained in one half of the bigon.
For the midpoint $X$ of the bigon, we have: $|XA|<\pi/2, |XB|<\pi/2, |AB|>\pi/2$. Then $\angle XAB$ and $\angle XBA$ are acute, so $X$ projects on $AB$. Then $SQD^{sph}_X$ has at least 3 minima. Since $X\in Core(P)$,  the function $SQD^{sph}_X$ has at least 3 short maxima, so $P$ is not skew.

    \end{proof}

    It remains to analyze triangles and quadrilaterals.
    \begin{enumerate}
        \item $P$ is a triangle. All its angles are obtuse, so it's not skew by Lemma \ref{lemSkewTria}.
        \item $P$ is a quadrilateral. Consider the triangle $KLM$ as in Fig. \ref{Fig2}. It is a connected component of $Big(P)\setminus P$. We have: $|KL| < \pi/2$, $|KM| < \pi/2$, $|LM|< \pi/2$, $\angle KLM <\pi/2$, $\angle KML <\pi/2$, $\angle LKM >\pi/2$. Then spherical geometry tells us that $K$ projects to $LM$. Let $KH \perp LM$. Moreover, $\angle LKH<\pi/2$, $\angle MKH<\pi/2$. Let $S = KH \cap RT$. Then $|SR|<\pi/2$, $|ST|<\pi/2$. Therefore $S\in Core(P)$, and $SQD^{sph}_S$ has at least three (short) minima at $R,T,$ and $H$. Therefore there are three short maxima, so $P$ is not skew.

    \end{enumerate}  Proposition \ref{propAcuteBigon} is proven.
\end{proof}

\begin{figure}[h]
  \includegraphics[width=0.4\linewidth]{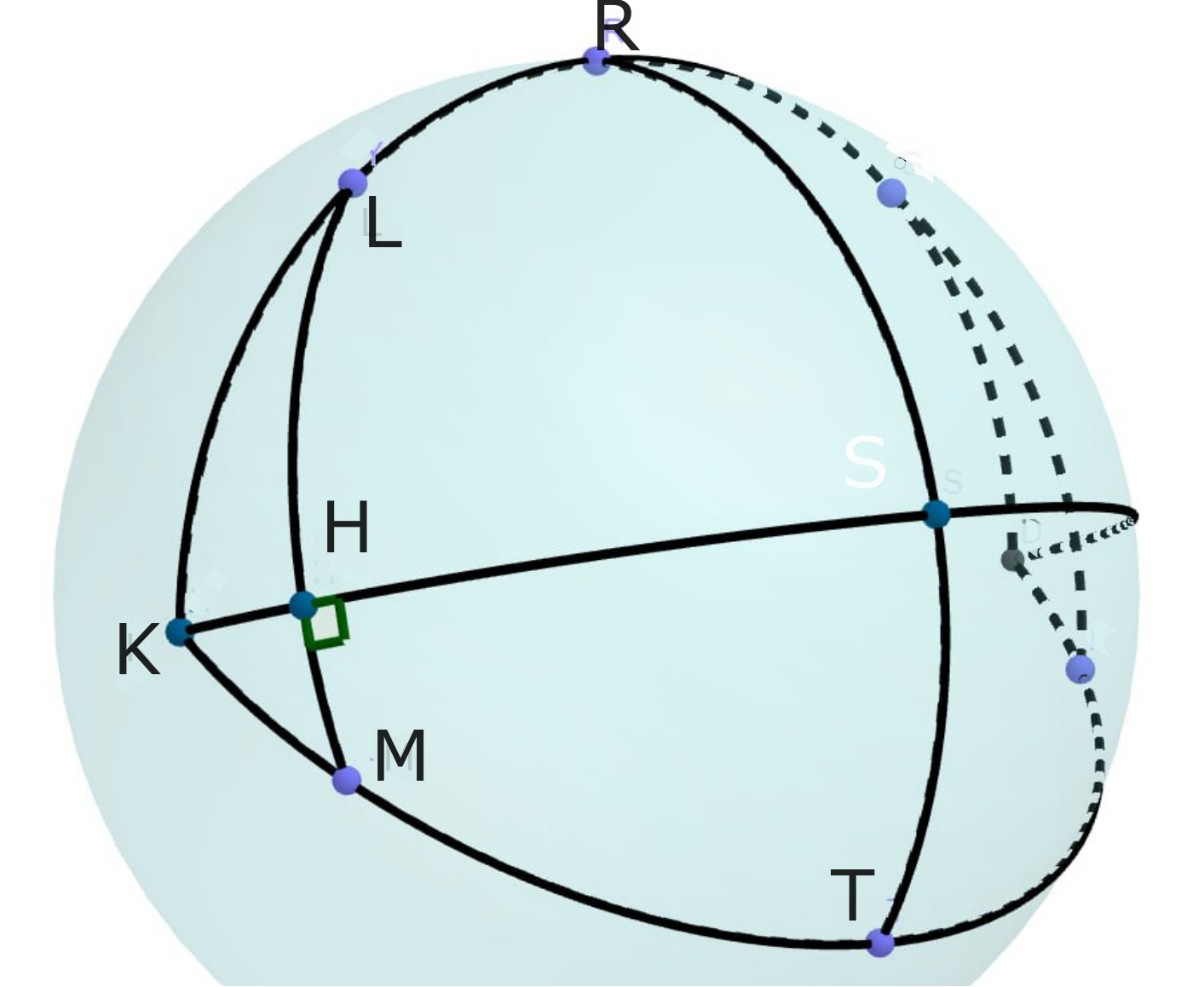}
  \caption {Case 2: $P$ is a quadrilateral. }
  \label{Fig2}
\end{figure}

  Proposition \ref{propAcuteBigon}  implies that the   midline $RT$ is contained in $Core(P)$. With this knowledge we proceed with the proof of Proposition \ref{propSkew}. So let $X$ be the midpoint of the minimal bigon $Big(P)$.

Denote by $M_1$ and $M_2$ maxima points of $SQD^{sph}_X$. They lie in different halves of the bigon.

\begin{lemma}\label{twoAcuteAngles}
$M_1$ and $M_2$ are acute vertices.
\end{lemma}

\begin{proof}
Assume  that $M_1$ obtuse.
    Two cases are possible:
   
    \begin{enumerate}
        \item $M_1$ lies on an edge of the bigon. $M_1$ is clearly not a vertex of bigon  since the bigon is acute-angled.
        Let a point $X$ travels along the midline in the direction of the edge  containing $M_1$. The points $R$ and $T$ are short minima of $SQD^{sph}_X$ all the time. Since these minima are fixed, and we have not other minima, there are no bifurcations on the way. But if $X$ is close to $R$, and $M_1$ is an obtuse vertex, it can not be maximum of $SQD^{sph}_X$.

        \item $M_1$ lies in the interior of the bigon. Then incident edges $M_1A$ and $M_1B$ are short, and angle between them is obtuse. Consider the point $Z$ which is defined as the intersection of two perpendiculars, as in Fig. \ref{Fig3}. Clearly,  $|ZB|<\pi/2, |ZM_1|<\pi/2, |ZA|<\pi/2, Z\in \mathcal{AR}^{spher}(M_1)$. Also $X \in \mathcal{AR}^{spher}(M_1)$. We also have $Z\in conv(A,B,M_1,X)$. Then $Z\in P$, and $P$ is not skew.

    \end{enumerate}
\end{proof}

    \begin{figure}[h]
  \includegraphics[width=0.4\linewidth]{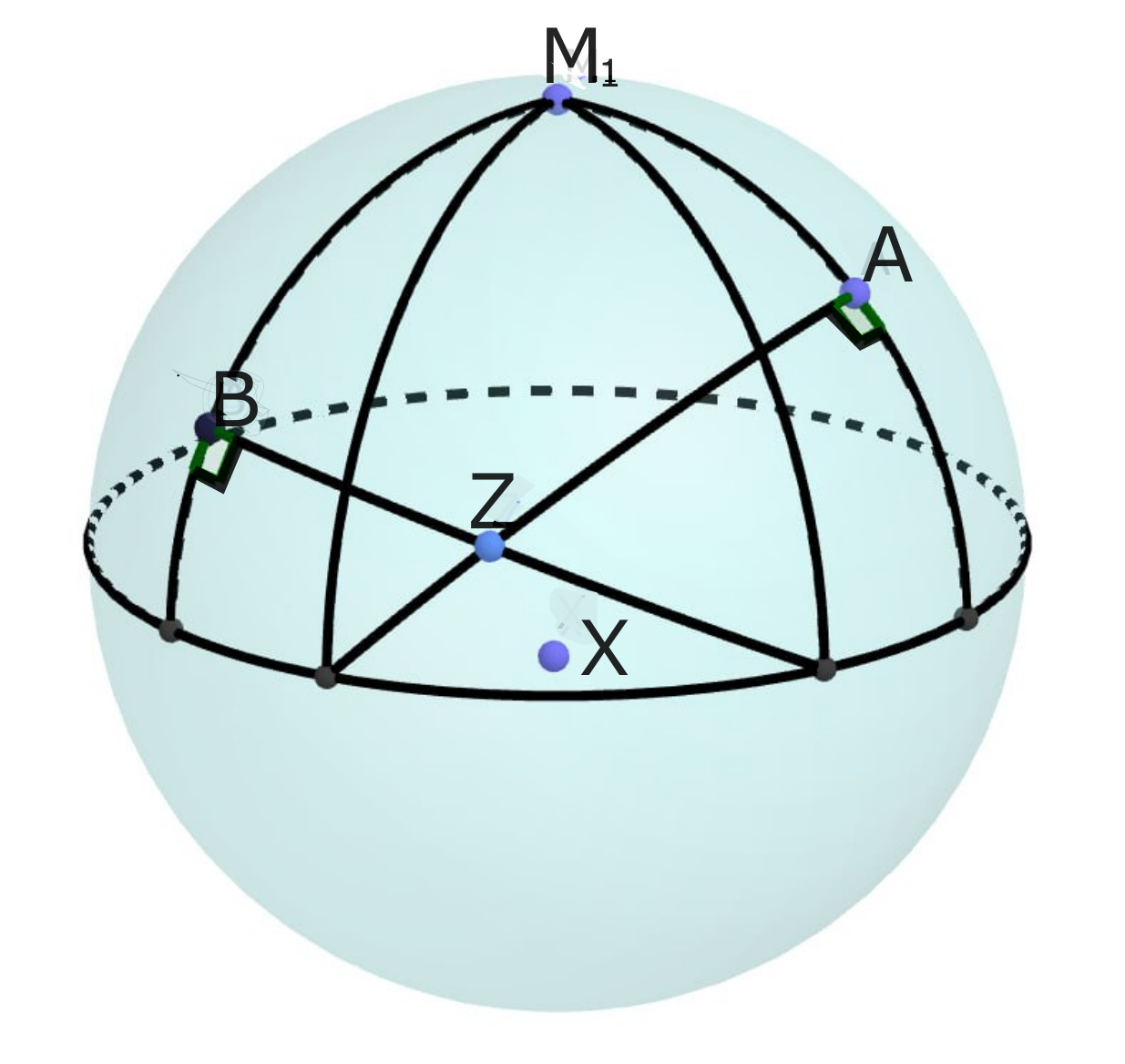}
  \caption {For the proof of Lemma \ref{twoAcuteAngles}}
  \label{Fig3}
\end{figure}

Clearly, there are no more acute vertices. Indeed, if a vertex is acute, it's a short maximal point of $SQD^{sph}_X$. Proposition \ref{propSkew}, (1) is proven.

\begin{lemma}\label{onetimes}
    Let $A$ be a vertex of $P$, $p\in P$ and $|M_1p|<\pi/2$, $|M_2p|<\pi/2$ ($M_1$ and $M_2$ are  the acute vertices). Then $|Ap|<\pi/2$.
\end{lemma}

\begin{proof}
     Assume that $A$ lies in the same half of bigon as $M_1$. Then we have:$|M_1p|<\pi/2$, $|M_1A|\leqslant \pi/2$, $\angle AM_1p<\pi/2$, which implies $|Ap|<\pi/2$.
\end{proof}

The following lemma completes the proof  of  Proposition \ref{propSkew}.

\begin{lemma}
    $|M_1M_2|>\pi/2$.
\end{lemma}

\begin{proof}
    Assume the contrary, that is, $|M_1M_2|<\pi/2$. Apply twice Lemma \ref{onetimes} for $p = M_1$ and $p = M_2$ one by one. Then for any vertex $A$ we have: $|M_1A|<\pi/2, |M_2A|<\pi/2$. Select a point $A'$ near a vertex $A$, inside its spherical active region. So, $SQD^{sph}_{A'}$ has three short maxima at $M_1, M_2,$ and  $ A$. Therefore $P$ is not skew. A contradiction.
\end{proof}



\begin{thebibliography}{99}


\bibitem{Agr} A.Agrachev, D. Pallaschke, S. Scholtes, \textit{On Morse theory for piecewise smooth functions}, Journal of Dynamical and Control Systems 3, 4 (1997), 449-469.


\bibitem{GrPanina} A. Grebennikov, G. Panina,  {\em A note on the concurrent normal conjecture}, Acta Mathematica
Hungarica, 167 (2022), 529-532.

\bibitem{Heil} E. Heil, {\em Concurrent normals and critical points under weak smoothness assumptions},
In: Discrete geometry and convexity (New York, 1982), volume 440 of Ann. New York
Acad. Sci., 170-178. New York Acad. Sci., New York, 1985.



\bibitem{Milnor}  J. Milnor, {\em  Morse theory}, Annals of Math. Studies 51, Princeton University Press, Princeton, 1963.


\bibitem{NasPanSiersma} I. Nasonov, G. Panina, D. Siersma, \textit{Concurrent normals problem for convex polytopes and Euclidean distance degree},
to appear in Acta Math. Hung.


\bibitem{Pardon} J. Pardon, {\em Concurrent normals to convex bodies and spaces of Morse functions}, Math. Ann.,
352(2012), no. 1, 55-71.








\end{thebibliography}
\end{document}